\documentclass[french,a4paper]{smfart}

\usepackage[T1]{fontenc}
\usepackage[latin1]{inputenc}
\usepackage[all]{xy}
\usepackage{smfthm}
\usepackage[english,frenchb]{babel}

\newtheorem*{bigtheo}{Théorème}

\author{Stéphane Bijakowski}
\address{Université Paris 13, Sorbonne Paris Cité,
LAGA,
CNRS, UMR 7539,
F-93430, Villetaneuse, France}
\email{bijakows@math.univ-paris13.fr}

\title{Classicité de formes modulaires de Hilbert}

\begin{document}

\frontmatter

\begin{abstract}
Nous prouvons un résultat de classicité pour les formes modulaires de Hilbert surconvergentes. Nous utilisons pour démontrer ce résultat la méthode du prolongement analytique, initialement développée par Buzzard et Kassaei.
\end{abstract}

\begin{altabstract}
We prove in this paper a classicality result for overconvergent Hilbert modular forms. To get this result, we use the analytic continuation method, first used by Buzzard and Kassaei.
\end{altabstract}

\maketitle

\tableofcontents

\mainmatter

\section*{Introduction}

Coleman (\cite{Co}) a prouvé qu'une forme modulaire surconvergente sur la courbe modulaire, de niveau iwahorique en $p$, de poids $k$ entier, propre pour un certain opérateur de Hecke $U_p$, était classique si la pente, c'est-à-dire la valuation de la valeur propre pour l'opérateur de Hecke (normalisée de telle façon que $v(p)=1$), était inférieure à $k-1$. Ce résultat a été obtenu grâce à une connaissance approfondie de la cohomologie rigide de la courbe modulaire. Des travaux de Buzzard (\cite{Bu}) et de Kassaei (\cite{Ka}), utilisant des techniques de prolongement analytique, ont donné une nouvelle démonstration de ce théorème. Ainsi, Buzzard a étudié la dynamique de l'opérateur de Hecke, et montré que ses itérés accumulaient le tube supersingulier dans un voisinage strict arbitrairement petit du lieu ordinaire multiplicatif. Une forme modulaire surconvergente étant définie sur un tel voisinage strict, l'équation $f = a_p^{-1} U_p f$ permet de prolonger $f$ au tube supersingulier dès que la pente est non nulle. \\
\indent La théorie du sous-groupe canonique a ensuite permis à Kassaei de décomposer l'opérateur de Hecke sur le lieu ordinaire-étale (et même sur un voisinage strict de celui-ci) en $U_p = U_p^{good} + U_p^{bad}$, où $U_p^{good}$ paramètre les supplémentaires du sous-groupe universel ne rencontrant pas le sous-groupe canonique, et $U_p^{bad}$ l'unique supplémentaire égal au sous-groupe canonique. L'opérateur $U_p^{good}$ est à valeurs dans un voisinage strict du tube ordinaire-multiplicatif, donc agit sur les formes modulaires surconvergentes, alors que $U_p^{bad}$ stabilise le lieu ordinaire-étale. Kassaei a alors défini une série $f_n$, approximant $a_p^{-n} U_p^n f$. Plus précisément, on a, lorsque cela a un sens, $f_n = a_p^{-n} U_p^n f - a_p^{-n} (U_p^{bad})^n f$. La condition $v(a_p) < k-1$ implique alors que la norme de l'opérateur $a_p^{-1} U_p^{bad}$ est strictement inférieure à $1$, et donc que la série $f_n$ converge sur le lieu ordinaire-étale. Cela permet donc d'étendre $f$ à toute la courbe modulaire. \\
\indent Ce résultat de classicité a été entendu aux formes modulaires de Hilbert par plusieurs auteurs. La méthode originale de Coleman, qui étudie la cohomologie de la courbe modulaire, a récemment été utilisée par Johansson (\cite{Jo}) ainsi que par Tian et Xiao (\cite{T-X}). La méthode du prolongement analytique a été utilisée par Sasaki dans le cas où le nombre premier $p$ est totalement décomposé (\cite{Sa}), Tian (\cite{Ti}) pour certains cas où $p$ est inerte, et Pilloni-Stroh (\cite{P-S 1}) pour le cas inerte général. Remarquons également que la méthode de prolongement analytique a été utilisée pour des variétés plus générales dans \cite{P-S 2} et dans \cite{Bi}, respectivement dans le cas totalement décomposé et inerte.

Nous étendons ici la méthode du prolongement analytique au cas général des formes modulaires de Hilbert. Soit $F$ un corps totalement réel de degré $d$, et $(p)=\prod \pi_i^{e_i}$ la décomposition de l'idéal engendré par $p$ dans $O_F$, l'anneau des entiers de $F$. On note également $f_i$ le degré résiduel de $\pi_i$. Soit $\Sigma_i = \{ \sigma \in $ Hom$(F,\mathbb{C}_p), v(\sigma(\pi_i)) > 0 \}$ ; les ensembles $\Sigma_i$ forment une partition de $\Sigma =$ Hom$(F,\mathbb{C}_p)$, et sont de cardinal $e_i f_i$. Le poids d'une forme modulaire de Hilbert est alors un caractère de Res$_{F / \mathbb{Q}} \mathbb{G}_m$, que l'on peut voir comme un élément de $\mathbb{Z}^{\Sigma}$. Nous avons alors le théorème suivant :

\begin{bigtheo}
Soit $f$ une forme de Hilbert surconvergente de niveau Iwahorique en $p$, de poids $\kappa=(k_\sigma)$, où $\sigma$ parcours l'ensemble $\Sigma$. Supposons que $f$ soit propre pour les opérateurs de Hecke $U_{\pi_i}$, de valeurs propres $a_i$, et que l'on ait pour tout $i$
$$e_i(v(a_i) + f_i) < \inf_{\sigma \in \Sigma_i} k_\sigma$$
Alors $f$ est classique.
\end{bigtheo}

Parlons à présent de l'organisation du texte. Dans la partie $\ref{varhilbert}$, nous introduisons la variété de Hilbert, ainsi que les formes modulaires sur cette variété. Dans la partie $\ref{hecke}$, nous introduisons les opérateurs de Hecke, et introduisons des décompositions de ces opérateurs sur certaines zones. Nous démontrons le théorème de classicité dans la partie $\ref{classicite}$, et la partie $\ref{compact}$ est consacrée à la construction de compactifications toroïdales, et au principe de Koecher.

\section{Variété et formes de Hilbert} \label{varhilbert}

\subsection{L'espace de modules}

Soit $F$ un corps totalement réel de degré $d \geq~2$ ; on note $O_F$ son anneau des entiers, $O_F^\times$ le groupe des unités et $O_F^{\times,+}$ le sous-groupe des unités totalement positives. Soit $p$ un nombre premier et $(p) = \prod_{i=1}^g\pi_i^{e_i}$ sa décomposition en idéaux premiers dans $F$. On notera $f_i$ le degré résiduel de $\pi_i$, et $\pi = \prod_i \pi_i$. Soit $F_{\pi_i}$ la complétion de $F$ suivant l'idéal $\pi_i$, et $K$ une extension finie de $\mathbb{Q}_p$ contenant la clôture normale des $F_{\pi_i}$. On notera $O_K$ l'anneau des entiers de $K$. Un schéma abélien de Hilbert-Blumenthal (SAHB) sur un schéma $S$ est un schéma abélien $A$ sur $S$ de dimension $d$ muni d'un plongement $O_F \hookrightarrow $ End$(A)$. Soit $N \geq 3$ un entier premier à $p$.
\begin{defi}
Soit $\delta$ la différente de $F$, $\mathfrak{c}$ un idéal fractionnaire de $F$. On note $\mathfrak{c}^+$ le cône des éléments totalement positifs. Soit $Y_{\mathfrak{c}} \to $ Spec $O_K$ l'espace de modules dont les $S$-points sont les classes d'isomorphismes des $(A,i,\phi,H)$ avec :
\begin{itemize}
\item $A \to S$ un SAHB.
\item $i : \delta^{-1} \otimes_{\mathbb{Z}} \mu_N \hookrightarrow A[N]$ est une structure de niveau $\mu_N$.
\item $\phi$ est une $\mathfrak{c}$-polarisation, c'est-à-dire que $\phi$ est un homomorphisme $O_F$-linéaire $\mathfrak{c} \to P(A)$, où $P(A)$ est l'ensemble des morphismes symétriques $f : A \to A^t$, tel que
\begin{itemize}
\item $\phi$ envoie $\mathfrak{c}^+$ dans le cône des polarisations
\item $\phi$ induit un isomorphisme $A \otimes \mathfrak{c} \simeq A^t$. 
\end{itemize}
\item $H$ est un sous-groupe de rang $p^f$ de $A[\pi]$ stable par $O_F$, avec $f=\sum_i f_i$, tel que $H[\pi_i]$ soit de rang $p^{f_i}$ pour tout $i$.
\end{itemize}
\end{defi}

	Comme $A[\pi] = \oplus_i A[\pi_i]$, on a $H = \oplus_i H_i$, avec $H_i = H [\pi_i]$. Pour tout $i$, $H_i$ est un sous-groupe de rang $p^{f_i}$ de $A[\pi_i]$, avec une action de $O_F / \pi_i \simeq \mathbb{F}_{p^{f_i}}$ : c'est un schéma en $\mathbb{F}_{p^{f_i}}$-vectoriel de rang $1$, c'est-à-dire un schéma en groupes de Raynaud. \\
D'après \cite{D-P}, on dispose d'un ouvert $Y_{\mathfrak{c}}^R \hookrightarrow Y_\mathfrak{c}$ qui est le lieu où le faisceau conormal $\omega_A$ du SAHB universel est un $O_F \otimes_\mathbb{Z} \mathcal{O}_{Y_\mathfrak{c}}$-module libre de rang $1$ (c'est le lieu de Rapoport). Le complémentaire de cet ouvert est un fermé de codimension plus grande que $2$ dans $Y_\mathfrak{c}$. \\
	Pour définir les formes modulaires entières de poids général, nous aurons besoin de modifier la fibre spéciale de $Y_\mathfrak{c}$. Nous nous inspirons de \cite{Sa2}. Rappelons que le faisceau conormal de $A$ en sa section unité est un $O_F \otimes _{\mathbb{Z}} \mathcal{O}_{Y_\mathfrak{c}}$-module. Commençons par décrire la $O_K$-algèbre $O_F \otimes_{\mathbb{Z}} O_K$. \\
Notons $F_{\pi_i}^{nr}$ l'extension maximale non ramifiée de $\mathbb{Q}_p$ contenue dans $F_{\pi_i}$. Soit également $\Sigma_i =Hom(F_{\pi_i},\overline{K})$ et $S_i = Hom(F_{\pi_i}^{nr},\overline{K})$. Soit $\varpi_i$ une uniformisante de $F_{\pi_i}$, et $E_i(u)$ le polynôme minimal de $\varpi_i$ sur $F_{\pi_i}^{nr}$ (c'est un polynôme d'Eisenstein de degré $e_i$). Pour $\sigma \in S_i$, on notera $E_\sigma (u) = \sigma E_i (u)$ ; c'est un polynôme à coefficients dans $O_K$. Alors on a 
$$O_F \otimes_\mathbb{Z} O_K = \bigoplus_{i=1}^g O_{F_{\pi_i}} \otimes_{\mathbb{Z}_p} O_K = \bigoplus_{i=1}^g (O_{F_{\pi_i}}^{nr} \otimes_{\mathbb{Z}_p} O_K) [u] / E_i(u) = \bigoplus_{i=1}^g \bigoplus_{\sigma \in S_i} O_K [u] / E_\sigma(u)$$
Si $S$ est un $O_K$-schéma, et $A \to S$ un SAHB, alors on peut décomposer le faisceau $\omega_A$ en 
$$\omega_A = \bigoplus_{i=1}^g \bigoplus_{\sigma \in S_i} \omega_{A,\sigma}$$
où $\omega_{A,\sigma}$ est un $\mathcal{O}_S$-module localement libre de rang $e_i$ muni d'une action de $O_K [u] / E_\sigma(u)$, pour tout $i$ et $\sigma \in S_i$. \\
Pour $\sigma \in S_i$, notons $\sigma_1, \dots, \sigma_{e_i}$ les éléments de $\Sigma_i$ dont la restriction à $F_{\pi_i}^{nr}$ est $\sigma$.

\begin{defi}
Soit $X_\mathfrak{c}$ l'espace de modules sur $O_K$ dont les $S$-points sont les couples\\ $(A,i,\phi,H,(\omega_{A,\sigma,j})_{i,\sigma \in S_i, 0 \leq j \leq e_i})$ avec
\begin{itemize}
\item $(A,i,\phi,H) \in Y_\mathfrak{c}(S)$
\item Pour tout $i$ et $\sigma \in S_i$, le faisceau $\omega_{A,\sigma}$ est muni de la filtration
$$0 = \omega_{A,\sigma,0} \subset \omega_{A,\sigma,1} \subset \dots \subset \omega_{A,\sigma,e_i} = \omega_{A,\sigma}$$
tel que 
\begin{itemize}
\item pour tout $j$, $\omega_{A,\sigma,j}$ est localement un $\mathcal{O}_S$-facteur direct stable par $O_F$ de $\omega_{A,\sigma}$ de rang $j$.
\item pour tout $1 \leq j \leq e_i$, l'action de $O_F$ sur  $\omega_{A,\sigma,j} / \omega_{A,\sigma,j-1}$, qui est un $\mathcal{O}_S$-module localement libre de rang $1$, se factorise par $O_F \to O_{F_{\pi_i}} \overset{\sigma_j}{\to}~O_K \to~\mathcal{O}_S$.
\end{itemize}
\end{itemize}
\end{defi}
Le foncteur $X_\mathfrak{c}$ est représentable par un schéma, que l'on notera encore $X_\mathfrak{c}$ (\cite{Sa2}). On dispose d'un morphisme d'oubli $X_\mathfrak{c} \to Y_\mathfrak{c}$, qui est surjectif. Au-dessus du lieu de Rapoport, c'est un isomorphisme ; en particulier, on a $X_\mathfrak{c} \otimes_{O_K} K \simeq Y_\mathfrak{c} \otimes_{O_K} K$. \\
Soit $Cl(F)^+$ le quotient des idéaux fractionnaires par les idéaux engendrés par les éléments totalement positifs, et $\{\mathfrak{c}_i\}$ un ensemble de représentants premiers à $p$. On note $X = \coprod_i X_{\mathfrak{c}_i}$. On notera $X_K = X \times K$, $\mathfrak{X}$ la complétion formelle de $X$ le long de sa fibre spéciale, et $X_{rig}$ la fibre générique rigide de $\mathfrak{X}$. 

\begin{rema}
Le schéma $X_\mathfrak{c}$ diffère de $Y_\mathfrak{c}$ par un éclatement. En effet, le morphisme $X_\mathfrak{c} \to Y_\mathfrak{c}$ est birationnel (c'est un isomorphisme au-dessus du lieu de Rapoport), et d'après \cite{Da}, un morphisme projectif birationnel entre schémas quasi-projectifs intègres noethériens est un éclatement. De plus, le sous-schéma fermé relatif à cet éclatement est contenu dans le complémentaire du lieu de Rapoport, donc en particulier dans la fibre spéciale de $Y_\mathfrak{c}$. Si on note $X_{\mathfrak{c},rig}$ et $Y_{\mathfrak{c},rig}$ les espaces rigides associés respectivement à $X_\mathfrak{c}$ et $Y_\mathfrak{c}$, on en déduit donc que $X_{\mathfrak{c},rig} \simeq Y_{\mathfrak{c},rig}$. 
\end{rema}

\begin{rema}
Nous aurions pu ajouter dans la définition de $X_\mathfrak{c}$ une condition de filtration pour le faisceau $\omega_H$ (ce qui est fait dans \cite{Sa2}). Nous aurions obtenu un espace plus régulier, mais notre définition est suffisante dans notre cadre.
\end{rema}

Dans \cite{Fa}, Fargues a défini une fonction degré pour les schémas en groupes finis et plats définis sur l'anneau des entiers d'une extension finie de $\mathbb{Q}_p$. Nous utilisons cette fonction pour décrire l'espace rigide $X_{rig}$.

\begin{defi}
On définit la fonction $Deg : X_{rig} \to \prod_{i=1}^g [0,f_i]$ par
$$Deg(A,i,\phi,H,\omega_{A,\sigma,j}) = (\text{deg } H_i)_{1 \leq i \leq g}$$
où deg est la fonction définie par Fargues dans \cite{Fa}.
\end{defi}

On notera également $Deg_i : X_{rig} \to [0,f_i]$ la $i$-ème composante de la fonction $Deg$. Si $v=(v_1, \dots, v_g)$, avec $v_i \in [0,f_i]$, on note $X_v = $deg$^{-1} (v)$ et $X_{\geq v} = $deg$^{-1} (\prod_i [v_i,f_i])$. De même, si $I=\prod_{j=1}^g I_j$, où $I_j$ est un intervalle de $[0,f_j]$, on note $X_I = $deg$^{-1} I$.

\begin{prop}
Si $I$ est un produit d'intervalles, $X_I$ est un ouvert de $X_{rig}$. Si $I$ est de plus compact à bornes rationnelles, $X_I$ est quasi-compact.
\end{prop}

\begin{proof}
Voir \cite{Bi} proposition $1.13$.
\end{proof}

Le lieu ordinaire-multiplicatif est $X_{(f_1, \dots, f_g)}$ ; c'est le lieu où $H$ est de type multiplicatif.

\subsection{Formes modulaires de Hilbert}

Soit $\Sigma = $Hom$(F,\overline{K})$, et $\kappa = (k_\sigma) \in \mathbb{Z}^\Sigma$. Rappelons que pour $1 \leq i \leq g$, on a noté \\
$\Sigma_i = Hom(F_{\pi_i},\overline{K}) =  \{ \sigma \in \Sigma, v(\sigma(\pi_i)) > 0 \}$, $S_i = Hom(F_{\pi_i}^{nr},\overline{K})$, et $\sigma_1, \dots, \sigma_{e_i}$ les éléments de $\Sigma_i$ dont la restriction à $F_{\pi_i}^{nr}$ est $\sigma$, pour tout $\sigma \in S_i$.

\begin{defi}
On définit le faisceau inversible $\omega^\kappa$ sur $X$ par
$$\omega^\kappa = \bigotimes_{i=1}^g \bigotimes_{\sigma \in S_i} \bigotimes_{j=1}^{e_i} (\omega_{A,\sigma,j} / \omega_{A,\sigma,j-1})^{k_{\sigma_j}}$$
\end{defi}

\begin{defi}
Une forme de Hilbert de poids $\kappa$ à coefficients dans une $O_K$-algèbre $C$ est un élément de $H^0(X \times $ Spec $ C,\omega^\kappa)$. 
\end{defi}

\begin{rema}
Au-dessus du lieu de Rapoport, le faisceau $\omega_A$ est un $O_F \otimes_{\mathbb{Z}} \mathcal{O}_{X}$-module libre de rang $1$. Si $U$ est un ouvert de ce lieu, on peut voir une forme modulaire $f$ à coefficients dans $O_K$ comme une loi fonctorielle, qui à un $R$-point $(A,i,\phi,H)$ de $U$ (on omet la filtration de $\omega_A$, qui est canonique) et une trivialisation $\omega : (R \otimes_{\mathbb{Z}}~O_F) \simeq~\omega_A$ associe un élément $f(A,i,\phi,H,\omega) \in R$ tel que pour tout $\lambda \in (R \otimes_{\mathbb{Z}} O_F)^\times$
$$f(A,i,\phi,H,\lambda \omega) = \lambda^{-\kappa} f(A,i,\phi,H,\omega)$$
où $\lambda^{-\kappa} = \prod_{\sigma \in \Sigma} \sigma(\lambda)^{-k_\sigma}$. \\
Si $U$ est un ouvert quelconque de $X$, on peut voir une forme modulaire $f$ à coefficients dans $O_K$ comme une loi fonctorielle, qui à un $R$-point $(A,i,\phi,H,\omega_{A,\sigma,j})$ et des isomorphismes de $R$-modules $\omega_\sigma : R^{e_i} \simeq \omega_{A,\sigma}$ respectant la filtration de $\omega_{A,\sigma}$ pour tout $\sigma \in S_i$, associe un élément $f(A,i,\phi,H,\omega_{A,\sigma,j},\omega_\sigma) \in R$, tel que pour tout couple $(l_\sigma)$, où $l_\sigma$ est une matrice triangulaire supérieure de $GL_{e_i}(R)$ pour $\sigma \in S_i$, on ait
$$f(A,i,\phi,H,\omega_{A,\sigma,j},l_\sigma \omega_\sigma) = (\prod_{i=1}^g \prod_{\sigma \in S_i} \prod_{j=1}^{e_i} \lambda_{\sigma_j}^{-k_{\sigma_j}}) f(A,i,\phi,H,\omega_{A,\sigma,j},\omega_\sigma)$$
où les $(\lambda_{\sigma_j})$ sont les coefficients diagonaux de $l_\sigma$.
\end{rema}

On note encore $\omega^\kappa$ l'analytifié de de ce faisceau, qui est un faisceau sur $X_{rig}$. Nous montrerons dans la partie $\ref{compact}$ que, par GAGA et le principe de Koecher, on a
$$H^0(X_K,\omega^{\kappa}) = H^0(X_{rig},\omega^{\kappa})$$

\begin{defi}
L'espace des formes modulaires surconvergentes de poids $\kappa$ est défini comme
$$H^0(X_K,\omega^\kappa)^\dagger := \text{colim}_\mathcal{V} H^0 (\mathcal{V},\omega^\kappa)$$
où la colimite est prise sur les voisinages stricts $\mathcal{V}$ du lieu ordinaire-multiplicatif $X_{(f_1, \dots, f_g)}$ dans $X_{rig}$.
\end{defi}

On dispose d'une application de restriction $H^0(X_K,\omega^{\kappa}) \to H^0(X_K,\omega^\kappa)^\dagger$. Cette application est injective, et l'image est l'ensemble des formes classiques.

\subsection{Normes}

Nous souhaitons définir une norme sur l'espace des formes modulaires. Soit $\mathcal{U}$ un ouvert de $X_{rig}$ et $f \in H^0(\mathcal{U},\omega^\kappa)$. Le faisceau $\omega^\kappa$ étant inversible sur $\mathcal{U}$, on peut définir comme dans \cite{Ka} une élément $|f(x)|$ pour tout $x \in \mathcal{U}$. Rappelons brièvement comment procéder. Soit $x \in \mathcal{U}$ un $L$-point, où $L$ est une extension finie de $K$. On a donc un morphisme $x : $ Spec $ L \to \mathcal{U}$, qui provient d'un unique morphisme $\tilde{x} : $ Spf $ O_L \to \mathfrak{X}$. Alors
$$H^0(\text{Spec } L , x^* \omega^\kappa) = H^0(\text{Spf } O_L,\tilde{x}^* \omega^\kappa) \otimes_{O_L} L$$
On définit alors une norme $| \cdot |_x$ sur $H^0($Spec $ L,x^* \omega^\kappa)$ en identifiant $H^0($Spf $ O_L, \tilde{x}^* \omega^\kappa)$ et les éléments de norme plus petite que $1$. Alors on définit $|f(x)| := |x^* f |_x$.

\begin{defi}
La norme de f sur $\mathcal{U}$ est définie comme
$$|f|_\mathcal{U} := \sup_{x \in \mathcal{U}} |f(x)|$$
\end{defi}

Cet élément est a priori infini, mais est fini si $\mathcal{U}$ est quasi-compact. On notera $\tilde{\omega}^\kappa$ le faisceau des fonctions de norme plus petite que $1$. Rappelons également le lemme suivant, dû à Kassaei (\cite{Ka}), qui prouve qu'une forme modulaire est définie par ses réductions modulo $p^n$ pour tout $n$.

\begin{lemm} \label{glue}
Soit $\mathcal{U}$ un ouvert quasi-compact de $X_{rig}$. On a :
$$H^0(\mathcal{U},\omega^\kappa) \simeq H^0(\mathcal{U},\tilde{\omega}^\kappa) \otimes_{O_K} K \simeq \left( \underset{\leftarrow}{\lim} \text{ } H^0(\mathcal{U},\tilde{\omega}^\kappa / p^n) \right) \otimes_{O_K} K$$
\end{lemm}

\section{Opérateurs de Hecke} \label{hecke}

\subsection{Définition} \label{heckedef}

Soit $\mathfrak{c}$ un idéal fractionnaire, et $\mathfrak{m}$ un idéal premier de $O_F$. On considère la correspondance $C_{\mathfrak{c},\mathfrak{cm}}$ définie sur $K$ comme suit : c'est l'espace de modules dont les $R$-points sont les $(A,i,\phi,H,\omega_{A,\sigma,j},L)$, où $(A,i,\phi,H,\omega_{A,\sigma,j}) \in (X_\mathfrak{c} \times K) (R)$, et $L$ est un sous-groupe de rang $N(\mathfrak{m})$ de $A[\mathfrak{m}]$ stable par $O_F$ et disjoint de $H$. On dispose de deux projections $p_1$ et $p_2$, respectivement sur $X_\mathfrak{c} \times K$ et $X_\mathfrak{cm} \times K$. La première projection $p_1$ est l'oubli de $L$, et la projection $p_2$ est le quotient par $L$ : $p_2(A,i,\phi,H,\omega_{A,\sigma,j},L) = (A/L,i',\phi',H',\omega_{A,\sigma,j}')$, où $i'$, $\phi'$ sont les structures de niveau et polarisation induites par $i$ et $\phi$ sur $A/L$, et $H'$ est l'image de $H$ dans $A/L$. Rappelons que puisque nous travaillons sur $K$, la filtration $\omega_{A,\sigma,j}'$ est définie canoniquement.  \\
Notons pour tout $i$, $\sigma(i)$ l'unique élément tel que $\mathfrak{c}_i \mathfrak{m}$ soit égal à $\mathfrak{c}_{\sigma(i)}$ dans $Cl(F)^+$, c'est-à-dire qu'il existe un élément $x_i$ totalement positif avec $\mathfrak{c}_i \mathfrak{m} = x_i \mathfrak{c}_{\sigma(i)}$. L'élément $x_i$ est déterminé à un élément de $O_F^{\times,+}$ près, et nous fixerons le choix de ces éléments dans la suite. Le choix de $x_i$ donne un isomorphisme $X_{\mathfrak{c}_i \mathfrak{m}} \simeq X_{\mathfrak{c}_{\sigma(i)}}$. On peut donc voir la projection $p_2 : C_{\mathfrak{c}_i,\mathfrak{c}_i \mathfrak{m}} \to X_{\mathfrak{c}_i \mathfrak{m}} \times K$ comme une projection sur $X_{\mathfrak{c}_{\sigma(i)}} \times K$. On note $C_\mathfrak{m} = \coprod_i C_{\mathfrak{c}_i,\mathfrak{c}_i \mathfrak{m}}$ ; d'après ce qui précède, on a donc deux projections $p_1, p_2$ :~$C_\mathfrak{m} \to~X_K$. \\
On note encore $p_1$ et $p_2$ les morphismes $C_{\mathfrak{m},rig} \to X_{rig}$, où $C_{\mathfrak{m},rig}$ est l'espace rigide associé à $C_{\mathfrak{m}}$.

\begin{defi}
L'opérateur de Hecke géométrique $U_\mathfrak{m}$, défini sur les parties de $X_{rig}$, est défini par 
$$U_\mathfrak{m} (S) = p_2 p_1^{-1} (S)$$
Cette correspondance envoie les parties finies dans les parties finies, les ouverts zariskiens dans les ouverts zariskiens, et les ouverts admissibles quasi-compacts dans les ouverts admissibles quasi-compacts. 
\end{defi}

\begin{rema}
Pour définir l'opérateur $U_\mathfrak{m}$, nous avons eu besoin d'identifier $X_{\mathfrak{c}_i \mathfrak{m}}$ et $X_{\mathfrak{c}_{\sigma(i)}}$ à l'aide d'un élément totalement positif $x_i$ pour tout $i$, défini à une unité totalement positive près. La définition de cet opérateur dépend donc du choix de ces éléments.
\end{rema}

De même, on définit l'opérateur de Hecke agissant sur les formes modulaires. Rappelons que la projection $p_2 : C_{\mathfrak{m},rig} \to X_{rig}$ provient des morphismes $C_{\mathfrak{c}_i,\mathfrak{c}_i \mathfrak{m}} \to~X_{\mathfrak{c}_i \mathfrak{m}} \times~K$, composés avec les isomorphismes $X_{\mathfrak{c}_i \mathfrak{m}} \times K \simeq X_{\mathfrak{c}_{\sigma(i)}} \times K$, obtenus grâce aux éléments $x_i$. De même, l'élément $x_i$ induit un isomorphisme $H^0(X_{\mathfrak{c}_i \mathfrak{m}}, \omega^\kappa) \simeq H^0(X_{\mathfrak{c}_{\sigma(i)}}, \omega^\kappa)$. Cet isomorphisme envoie un élément $f \in H^0(X_{\mathfrak{c}_i \mathfrak{m}}, \omega^\kappa)$ vers la forme modulaire $g$ définie par $g(A,i,\phi,H,\omega_{A,\sigma,j},\omega_\sigma) = f(A,i,x_i \phi,H,\omega_{A,\sigma,j},\omega_\sigma)$. Les morphismes $H^0(X_{\mathfrak{c}_{\sigma(i)}} \times K, \omega^\kappa) \simeq H^0(X_{\mathfrak{c}_i \mathfrak{m}} \times K, \omega^\kappa) \to H^0(C_{\mathfrak{c}_i,\mathfrak{c}_i \mathfrak{m}},p_2^* \omega^\kappa)$ donnent donc un morphisme $H^0(X_K,\omega^\kappa) \to H^0(C_\mathfrak{m},p_2^* \omega^\kappa)$ (et un mophisme analogue pour les espaces rigides associés). \\
Soit $\pi_\mathfrak{m} : A \to A/L$ l'isogénie universelle sur $C_\mathfrak{m}$ ; elle induit un morphisme $\pi_\mathfrak{m}^* :~\omega_{A/L} \to~\omega_A$. Or au-dessus de $K$ on a la décomposition du faisceau $\omega_A$ en $\omega_A = \oplus_{\tau \in \Sigma} \omega_{A,\tau}$. Pour tout $\kappa \in \mathbb{Z}^\Sigma$, le morphisme $\pi_\mathfrak{m}^*$ induit donc un morphisme
$$\pi_\mathfrak{m}^\kappa : p_2^* \omega^\kappa \to p_1^* \omega^\kappa$$
Ce dernier induit donc un morphisme $H^0(\mathcal{V},p_2^* \omega^\kappa) \to H^0(\mathcal{V},p_1^* \omega^\kappa)$ pour tout ouvert $\mathcal{V}$ de $C_{\mathfrak{m},rig}$. 

\begin{defi}
Soit $\mathcal{U}$ un ouvert de $X_{rig}$. L'opérateur de Hecke $U_\mathfrak{m}$ agissant sur les formes modulaires de poids $\kappa$ est défini par le morphisme composé
\begin{displaymath}
U_\mathfrak{m} :   H^0(U_\mathfrak{m} (\mathcal{U}),\omega^\kappa) \to H^0 ( p_1^{-1} (\mathcal{U}), p_2^* \omega^\kappa) \overset{\pi_\mathfrak{m}^\kappa}{\to} H^0(p_1^{-1}(\mathcal{U}) , p_1^* \omega^\kappa) \overset{N(\mathfrak{m})^{-1} Tr_{p_1}}{\to}  H^0(\mathcal{U},\omega^\kappa)
\end{displaymath}
\end{defi}

\begin{rema}
Le terme $N(\mathfrak{m})^{-1}$ sert à normaliser l'opérateur de Hecke. Il maximise l'intégrabilité de cet opérateur, comme le montre un calcul sur les $q$-développements.
\end{rema}

\begin{rema}
Là encore, la définition de l'opérateur $U_\mathfrak{m}$ dépend des choix des éléments $x_i$. Néanmoins, si on se restreint aux formes $f$ invariantes sous l'action de $O_F^{\times,+}$ (c'est-à-dire telles que $f(A,i,\epsilon \phi,H,\omega_{A,\sigma,j},\omega_\sigma) = f(A,i,\phi,H,\omega_{A,\sigma,j},\omega_\sigma)$ pour tout $\epsilon \in O_F^{\times,+}$), alors cet opérateur est indépendant du choix des $x_i$ (voir \cite{P-S 3} paragraphe $6$ ou \cite{Pi2} paragraphe $1$). \\
De même, si on ne restreint pas aux formes invariantes par cette action, l'algèbre engendrée par les $U_\mathfrak{m}$, $\mathfrak{m}$ idéal premier divisant $\pi$, n'est plus nécessairement commutative.
\end{rema}

\begin{rema}
Si l'idéal $\mathfrak{m}$ est premier à $pN$, on note en général cet opérateur $T_\mathfrak{m}$. Néanmoins, nous utiliserons ces opérateurs uniquement pour $\mathfrak{m}$ divisant $p$, ce qui justifie notre notation.
\end{rema}

Dans la suite, nous nous restreindrons aux formes modulaires invariantes par l'action de $O_F^{\times,+}$. Par un abus de notation, on notera encore $H^0(\mathcal{U},\omega^\kappa)$ l'algèbre des formes modulaires $O_F^{\times,+}$-invariantes de poids $\kappa$ définies sur $\mathcal{U}$.  On dispose donc en particulier de $g$ opérateurs de Hecke $U_{\pi_1}, \dots, U_{\pi_g}$ agissant sur ces formes modulaires, et l'algèbre engendrée par ces opérateurs est commutative. 

\subsection{Propriétés}

Ces opérateurs se comportent bien relativement à la fonction Degré.

\begin{prop}
Soit $1 \leq i \leq g$, $x \in X_{rig}$ et $y \in U_{\pi_i} (x)$. Soit $(x_1, \dots, x_g)=~Deg(x)$, et $(y_1, \dots, y_g)=Deg(y)$. Alors
\begin{itemize}
\item $y_j=x_j$ pour $j \neq i$.
\item $y_i \geq x_i$
\end{itemize}
De plus, s'il existe $y \in U_{\pi_i}^{e_i} (x)$ avec $Deg(y)=Deg(x)$, alors $x_i \in \frac{1}{e_i} \mathbb{Z}$.
\end{prop}

\begin{proof}
Soit $x=(A,i,\phi,H,\omega_{A,\sigma,j})$ et $L$ le sous-groupe de $A[\pi_i]$ correspondant à $y$. Comme $L$ est disjoint de $A[\pi_j]$ pour $j \neq i$, on a un isomorphisme 
$A[\pi_j] \simeq (A/L)[\pi_j]$. Si on décompose $H$ en $\oplus_k H_k$ avec $H_k \subset A[\pi_k]$ pour $1 \leq k \leq g$, l'image de $H_j$ dans $(A/L)[\pi_j]$ est isomorphe à $H_j$, donc ont le même degré. Le premier point est donc vérifié. \\
De plus, $L$ est un supplémentaire générique de $H_i$ dans $A[\pi_i]$. Si on note $H_i'$ l'image de $H_i$ dans $(A/L)[\pi_i]$, on a alors un morphisme $H_i \to H_i'$, qui est un isomorphisme en fibre générique. D'après le corollaire $3$ de \cite{Fa}, on a deg $H_i \leq $deg $H_i'$, ce qui prouve le second point. \\
Supposons maintenant qu'il existe $y \in U_{\pi_i}^{e_i} (x)$ avec $Deg(y)=Deg(x)$. Soit $L$ le sous-groupe de $A[\pi_i^{e_i}]$ correspondant à $y$. On remarque alors que $H_i$ et $L$ sont en somme directe dans $A[\pi_i^{e_i}]$. En effet, l'image $H_i'$ de $H_i$ dans $A/L$ a le même degré que $H_i$ par hypothèse. Comme $H_i'=(H_i+L)/L$, on en déduit par additivité de la fonction degré que
$$\text{deg } (H_i + L) = \text{deg } H_{i}' + \text{deg } L = \text{deg } H_i + \text{deg } L$$
Le morphisme $H_i \times L \to H_i + L$ conserve donc le degré, et est un isomorphisme en fibre générique. D'après \cite{Fa}, c'est un isomorphisme. Les groupes $H_i$ et $L$ sont donc en somme directe dans $A[\pi_i^{e_i}]$. En particulier, on a $A[\pi_i]= H_i \oplus L[\pi_i]$. \\
Si $e_i$ était égal à $1$, $H_i$ serait un facteur direct de $A[\pi_i]$, donc de $A[p]$ qui est un $BT_1$, un Barsotti-Tate tronqué d'échelon $1$ (c'est-à-dire la $p$-torsion d'un groupe $p$-divisible). Le groupe $H_i$ serait donc également un $BT_1$, et son degré serait entier. \\
Malheureusement, en général $A[\pi_i]$ n'est pas un $BT_1$, on sait seulement que $A[\pi_i^{e_i}]$ en est un. Nous allons prouver que le sous-groupe $L$ en est un également. Soit $D$ le module de Dieudonné correspondant à $A[\pi_i^{e_i}] \times \mathbb{F}_q$ (avec $q=p^{f_i}$), que l'on décompose suivant l'action de $\mathbb{F}_q$ en $D = \oplus_{j=1}^{f_i} D_j$. On dispose du Frobenius $F_j : D_j \to D_{j+1}$ et du Verschriebung $V_j : D_{j+1} \to D_j$. Ces applications vérifient $F_j V_j=0$ et $V_j F_j = 0$. De plus, $A[\pi_i^{e_i}]$ étant un $BT_1$, on a Im $F_j = $ Ker $V_j$ et Ker $F_j =$ Im $V_j$. Nous allons montrer que $L$ est un $BT_1$. Cette propriété est équivalente à ce que $L \times \mathbb{F}_q$ soit un $BT_1$, soit que Im ${F_j}_{|L} = $ Ker ${V_j}_{|L}$ pour tout $j$. Or Ker ${V_j}_{|L} =$ Ker $V_j \cap L = $ Im $V_j \cap L$. Nous allons donc montrer que Im ${F_j}_{|L} = $ Im $F_j \cap L$ pour tout $j$. \\
Par souci de simplification des écritures, nous supprimons l'indice $j$ ; on a donc des applications $F,V : D \to D$, et $D$ est un $\mathbb{F}_q$-espace vectoriel de dimension $2e_i$. On choisit une base $(u_1, \dots, u_{e_i}, v_1, \dots, v_{e_i})$ de $D$ de telle sorte que $L$ corresponde à $(u_1, \dots, u_{e_i})$ et $H_i$ à $(v_1)$. Cela est bien possible, puisque $L$ et $H_i$ sont en somme directe dans $A[\pi_i^{e_i}]$. De plus, on suppose la base de $D$ choisie de telle sorte que la multiplication par $\pi_i$ envoie $u_k$ sur $u_{k-1}$ et $v_k$ sur $v_{k-1}$. Le fait que les espaces $(u_1, \dots, u_k)$ et $(u_1, \dots, u_k, v_1, \dots, v_k)$ soient stables par $F$, et les isomorphismes $A[\pi_i^{j+k}]/[\pi_i^j] \simeq A[\pi_i^k]$ montrent que la matrice de $F$ dans cette base est de la forme
\begin{displaymath}
\left(
\begin{array}{cccccccc}
x_1 & x_2 & \cdots & x_{e_i} & 0   & \epsilon_2 & \cdots & \epsilon_{e_i}  \\
    & x_1 & \ddots & \vdots  &     & 0          & \ddots & \vdots          \\
		&     & \ddots & x_2     &     &            & \ddots & \epsilon_2      \\
		&     &        & x_1	   &     &            &        & 0               \\
		&     &        &         & y_1 & y_2        & \cdots & y_{e_i}         \\ 
		&     &        &         &     & y_1        & \ddots & \vdots          \\ 
		&     &        &         &     &            & \ddots & y_2             \\ 
		&     &        &         &     &            &        & y_1                
\end{array}
\right)
\end{displaymath}
Soit $r$ l'entier positif tel que $x_1 = \dots= x_r=0$ et $x_{r+1} \neq 0$, et $s$ celui vérifiant $y_1 = \dots = y_s=0$ et $y_{s+1} \neq 0$. On voit alors que le noyau de $F$ est inclus dans $(u_1, \dots, u_{e_i},v_1, \dots, v_s)$, et est de dimension au plus $r+s$. Or d'après le théorème du rang, et par auto-dualité de $A[\pi_i^{e_i}]$, on a $2e_i = $ dim Ker $F + $ dim Im $F = $ dim~Ker~$F+$~dim Ker $V = 2 $ dim Ker $F$, soit dim Ker $F = e_i$. D'où $r+s \geq e_i$. \\
On voit de plus qu'il existe $w_1, \dots, w_{e_i-r}$ dans $(u_1, \dots, u_{e_i})$ tel que $(u_1, \dots, u_r, v_1+~w_1 , \dots, v_{e_i-r}+~w_{e_i-r})$ soit inclus dans le noyau de $F$. Par égalité des dimensions, c'est une égalité. Supposons maintenant que $s > e_i - r$. Alors $F v_{e_i-r+1} \in (u_1, \dots, u_{e_i-r})$ donc il existe $w_{e_i-r+1} \in (u_1, \dots, u_{e_i})$ tel que $v_{e_i-r+1} + w_{e_i-r+1}$ soit dans le noyau de $F$, ce qui est impossible par l'égalité précédente. D'où $s=e_i-r$. \\
Montrons maintenant que Im $F_{|L} = $ Im $F \cap L$. L'inclusion Im $F_{|L} \subset $ Im $F \cap L$ est évidente. Soit $x \in D$ tel que $Fx \in L$. Rappelons que $L=(u_1, \dots, u_{e_i})$. D'après ce qui précède, comme $y_{e_i-r+1} \neq 0$, $x \in (u_1, \dots, u_{e_i}, v_1, \dots, v_{e_i-r})$ donc $Fx \in~(u_1, \dots, u_{e_i-r}) =$~Im~$F_{|L}$. Nous avons donc montré que $L$ est un $BT_1$, donc en particulier son degré est entier. \\
Soit $L_k = L[\pi^k]$. On dispose d'isogénies $A \to A/L_1 \to \dots \to A/L_{e_i} = A/L$ ; si on note $H_{i,k}$ l'image de $H_i$ dans $A/L_k$, on voit que la suite $($ deg $H_{i,k})_k$ est croissante. Or par hypothèse $H_i$ et $H_{i,e_i}$ ont même degré. Les $H_{i,k}$ ont donc le même degré que $H_i$, et on a pour $0 \leq k < e_i$
$$(A/L_k) [\pi_i] = H_{i,k} \oplus L_{k+1}/L_k$$
donc deg $L_{k+1}/L_k = f_i - $ deg $H_i$. Cette quantité étant constante, on a donc deg $L_{k} = \frac{k}{e_i}$ deg $L$, et deg $H_i = f_i- $deg $L_1 = 1 - \frac{1}{e_i}$ deg $L \in \frac{1}{e_i} \mathbb{Z}$.
\end{proof}

\begin{rema}
La démonstration de la deuxième partie de la proposition est différente de celle dans \cite{Pi2}. Elle donne un résultat plus fort, mais se généralise moins facilement à des variétés plus générales que les variétés de Hilbert.
\end{rema}

On voit en particulier que les opérateurs $U_{\pi_i}$ stabilisent les espaces $X_{\geq v}$, pour $v=(v_j)$ et $v_j \in [0,f_j]$. Ils agissent donc sur les formes modulaires surconvergentes. \\
De plus, l'opérateur $U_{\pi_i}^{e_i}$ augmente strictement le degré de $x$ si $Deg_i(x) \notin \frac{1}{e_i} \mathbb{Z}$. En réalité, nous avons une proposition plus forte.

\begin{prop} \label{dyna}
Soit $1 \leq i \leq g$, $k$ un entier compris entre $0$ et $f_ie_i-1$, et $\alpha, \beta$ deux rationnels tels que $\frac{k}{e_i} < \alpha < \beta < \frac{k+1}{e_i}$. Posons $X_{i,\geq u}=Deg_i^{-1} ([u,f_i])$, pour tout réel $u$. Alors il existe un entier $N$ tel que 
$$U_{\pi_i}^N(X_{i,\geq \alpha}) \subset X_{i,\geq \beta}$$
\end{prop}

\begin{proof}
Supposons par l'absurde qu'il existe $x_n \in X_{i,\geq \alpha}$ et $y_n \in U_{\pi_i}^n(x_n)$ avec $Deg_i (y_n)<\beta$. D'après la proposition précédente, cela entraîne $x_n \in X_{i,[\alpha,\beta]}:=~Deg_i^{-1}([\alpha,\beta])$. Or cet espace est quasi-compact ; de plus la fonction $X_{rig} \to \mathbb{R}$ définie par $x \to~\inf_{y \in U_{\pi_i}^{e_i}(x)} Deg_i(y) -~Deg_i(x)$ est continue, et à valeurs strictement positives sur $X_{i,[\alpha,\beta]}$. On en déduit qu'elle y atteint son minimum, soit qu'il existe $\epsilon > 0$, avec
$$ Deg_i(y) \geq Deg_i(x) + \epsilon$$
pour tout $x \in X_{i,[\alpha,\beta]}$ et $y \in U_{\pi_i}^{e_i}(x)$. \\
Cela implique donc que $Deg_i(y_{ne_i}) \geq n\epsilon + Deg_i(x_{ne_i}) \geq n \epsilon + \alpha$ pour tout $n$, ce qui est impossible.
\end{proof}

\subsection{Décomposition des opérateurs de Hecke} \label{decompo}

Soit $\mathcal{U}$ un ouvert quasi-compact de $X_{rig}$. Fixons un élément $i$ compris entre $1$ et $g$, et un élément rationnel $r \in [0,f_i]$. On note $X_{i,\leq r} := \{ x \in X_{rig}, Deg_i (x) \leq r \}$. Nous voulons stratifier notre ouvert $\mathcal{U}$ suivant le nombre de points de $U_{\pi_i}(x) \cap X_{i,\leq r}$. Pour tout $x=(A,i,\phi,H,\omega_{A,\sigma,j}) \in X_{rig}$, soit $N(x,r)$ le nombre de points de $U_{\pi_i}(x) \cap X_{i,\leq r}$. Définissons
$$\mathcal{U}_j := \{ x \in \mathcal{U}, N(x,r) \geq j \}$$

\begin{prop}
Les $(\mathcal{U}_j)$ forment une suite décroissante d'ouverts quasi-compacts, vide à partir d'un certain rang.
\end{prop}

\begin{proof}
Voir \cite{Bi} lemme $2.7$.
\end{proof}

Sur $\mathcal{U}_j \backslash \mathcal{U}_{j+1}$, on a $N(x,r)=j$. On peut alors décomposer l'opérateur $U_{\pi_i}$ en $U_{\pi_i}^{good} \coprod U_{\pi_i}^{bad}$, où $U_{\pi_i}^{bad}$ correspond aux $j$ points de $X_{i,\leq r}$, et $U_{\pi_i}^{good}$ aux autres. Remarquons que $U_{\pi_i}^{bad}$ paramètre les supplémentaires $L$ de $H_i$ avec deg $L \geq f_i -r$. De plus, il est possible de faire surconverger ces ouverts.

\begin{prop}
Soit $r'>r$ un nombre rationnel, et $\mathcal{U}_j' := \{ x \in \mathcal{U}, N(x,r') \geq~j \}$. Alors $\mathcal{U}_j'$ est un voisinage strict de $\mathcal{U}_j$ dans $\mathcal{U}$, c'est-à-dire que le recouvrement $(\mathcal{U}_j',\mathcal{U} \backslash \mathcal{U}_j)$ de $\mathcal{U}$ est admissible.
\end{prop}

\begin{proof}
Voir \cite{Bi} proposition $2.10$.
\end{proof}

Pour $r'>r$, on dispose donc de la décomposition de $U_{\pi_i}$ sur $\mathcal{U}_j \backslash \mathcal{U}_{j+1}$, ainsi que sur $\mathcal{U}_j' \backslash \mathcal{U}_{j+1}'$. Ces décompositions coïncident sur l'intersection des deux ensembles. \\
Il est possible de généraliser cette décomposition à $U_{\pi_i}^N$ pour tout entier $N$.

\begin{theo} \label{bigtheo}
Soit $N \geq 1$ et $r \in [0,f_i]$ un rationnel. Il existe un ensemble fini totalement ordonné $S_N$ et une suite décroissante d'ouverts quasi-compacts $(\mathcal{U}_j (N))_{i \in S_N}$ de $\mathcal{U}$ de longueur $L=L(N)$ indépendante de $\mathcal{U}$, tels que pour tout $j\geq 0$, on peut décomposer la correspondance $U_{\pi_i}^N$ sur $\mathcal{U}_{j}(N) \backslash \mathcal{U}_{j+1}(N)$ en 
$$ U_{\pi_i}^N = \left ( \coprod_{k=0}^{N-1} U_{\pi_i}^{N-1-k} \circ T_k  \right ) \coprod T_N$$
avec $T_0 = U_{\pi_i,j,N}^{good}$, pour $0 < k < N$
$$T_k = \coprod_{j_1 \in S_{N-1}, \dots, j_k \in S_{N-k}}  U_{\pi_i,j_k,N}^{good} U_{\pi_i,j_{k-1},j_k,N}^{bad} \dots U_{\pi_i,j,j_1,N}^{bad}$$
et
$$T_N = \coprod_{j_1 \in S_{N-1}, \dots, j_{N-1} \in S_1} U_{\pi_i,j_{N-1},N}^{bad} U_{\pi_i,j_{N-2},j_{N-1},N}^{bad} \dots U_{\pi_i,j,j_1,N}^{bad}$$  
avec
\begin{itemize}
\item les images des opérateurs $U_{\pi_i,j,N}^{good}$ ($j \in S_k$) sont incluses dans $X_{i,\geq r}=\{ x \in~X_{rig}, Deg_i (x) \geq~r \}$
\item les opérateurs $U_{\pi_i,j,l,N}^{bad}$ ($j \in S_k$, $l \in S_{k-1})$ et $U_{\pi_i,j,N}^{bad}$ ($j \in S_1$) sont obtenus en quotientant par un sous-groupe $L$ de degré supérieur à $f_i-r$. 
\end{itemize}
Enfin, si $(\mathcal{U}_j'(N))$ est la stratification de $\mathcal{U}$ obtenue pour $r'>r$, alors $\mathcal{U}_j'(N)$ est un voisinage strict de $\mathcal{U}_j(N)$ dans $\mathcal{U}$ pour tout $j$.
\end{theo}

\begin{proof}
C'est le théorème $2.15$ de \cite{Bi}.
\end{proof}

\subsection{Normes}

Pour démontrer le théorème de classicité, nous aurons besoin d'un calcul de normes de ces opérateurs de Hecke. Rappelons que la norme d'un opérateur $T : H^0(T(\mathcal{U}),\mathcal{F}) \to H^0(\mathcal{U},\mathcal{F})$ est défini par
$$ \left\| T \right\|_{\mathcal{U}} := \inf \left\{ \lambda \in \mathbb{R}_{>0} , |Tf|_\mathcal{U} \leq \lambda |f|_{T(\mathcal{U})} \forall f \in H^0(T(\mathcal{U}),\mathcal{F}) \right\}$$

\begin{prop} \label{lemnorm}
Soit $T$ un opérateur défini sur un ouvert $\mathcal{U}$, égal à $U_{\pi_i}$, $U_{\pi_i}^{good}$ ou $U_{\pi_i}^{bad}$. On suppose que cet opérateur ne fait intervenir que des sous-groupes $L$ de $A[\pi_i]$ avec deg $L \geq c$, pour un certain $c \geq 0$. Alors
$$ \Vert T \Vert_\mathcal{U} \leq p^{f_i-c \inf_{\tau \in \Sigma_i} k_{\tau}}$$
\end{prop}

\begin{proof}
Soit $\mathfrak{m} =\pi_i$, $x=(A,i,\phi,H,\omega_{A,\sigma,j}) \in X_{rig}(\overline{K}) $ et $L \subset A[\mathfrak{m}]$ un supplémentaire de $H[\mathfrak{m}]$ stable par $O_F$. Alors le morphisme $\pi_\mathfrak{m} : A \to A/L$ donne une suite exacte de $O_{\overline{K}} \otimes_{\mathbb{Z}} O_F$-modules
$$0 \to \omega_{A/L} \overset{\pi_\mathfrak{m}^*}{\to} \omega_A \to \omega_L \to 0$$
En décomposant cette suite exacte selon les éléments de $S_k$ pour tout $1\leq k \leq g$, on obtient pour $\sigma \in S_k$
$$0 \to \omega_{A/L,\sigma} \overset{\pi^*_{\mathfrak{m},\sigma}}{\to} \omega_{A,\sigma} \to \omega_{L,\sigma} \to 0$$
De plus, $\pi^*_{\mathfrak{m},\sigma}$ est un isomorphisme si $\sigma \notin S_i$, et $\sum_{\sigma \in S_i} v(\det \pi^*_{\mathfrak{m},\sigma}) = $deg $L$. Puisque l'on travaille avec une variété abélienne définie sur $O_{\overline{K}}$, qui est plat sur $O_K$, le module $\omega_{A,\sigma}$ est canoniquement filtré par suivant les éléments de $\Sigma_i$ dont la restriction à $F_{\pi_i}^{nr}$ est $\sigma$. La filtration $(\omega_{A,\sigma,j}$ est donc canonique. Rappelons que $\omega_{A,\sigma,j}$ est un $O_{\overline{K}}$-module libre de rang $j$, qui est un facteur direct de $\omega_{A,\sigma}$, et tel que l'action de $O_F$ sur $\omega_{A,\sigma,j} / \omega_{A,\sigma,j-1}$ se factorise par $\sigma_j$. On obtient de même une filtration canonique $\omega_{A/L,\sigma,j}$ de $\omega_{A/L,\sigma}$. \\
Le morphisme $\pi^*_{\mathfrak{m},\sigma}$ respecte cette filtration ; on note $\lambda_{\sigma,j}$ le déterminant de\\ 
$\pi^*_{\mathfrak{m},\sigma,j} : \omega_{A/L,\sigma,j} / \omega_{A/L,\sigma,j-1} \to \omega_{A,\sigma,j} / \omega_{A,\sigma,j-1}$. Si $u_\sigma$ désigne le déterminant de $\pi^*_{\mathfrak{m},\sigma}$, on a donc 
$$u_\sigma = \prod_{j=1}^{e_i} \lambda_{\sigma,j}$$
En utilisant les notations de la partie $\ref{heckedef}$, on veut calculer la norme du morphisme $\pi_\mathfrak{m}^\kappa : p_2^* \omega^\kappa \to p_1^*\omega^\kappa$. Il suffit de calculer cette norme point par point. Or au-dessus de $x$, la norme de ce faisceau est égal à
$$\Vert \pi_\mathfrak{m}^\kappa \Vert_x = | \prod_{\sigma \in S_i} \prod_{j=1}^{e_i} \lambda_{\sigma,j}^{k_{\sigma_j}} |$$
En effet, calculer la norme de ce morphisme au-dessus de $x$ revient à calculer la norme du morphisme induit entre les faisceaux définis sur le schéma formel $\mathfrak{X}$ au-dessus de $\tilde{x}$. D'après ce qui précède, ce dernier morphisme est la multiplication par $\prod_{\sigma \in S_i} \prod_{j=1}^{e_i} \lambda_{\sigma,j}^{k_{\sigma_j}}$. \\
On en déduit que 
$$\Vert \pi_\mathfrak{m}^\kappa \Vert_x = p^{ - \sum_{\sigma \in S_i} \sum_{j=1}^{e_i} v(\lambda_{\sigma,j}) k_{\sigma_j} } \leq  p^{ - \inf_{\tau \in \Sigma_i} k_{\tau} \sum_{\sigma \in S_i} \sum_{j=1}^{e_i} v(\lambda_{\sigma,j}) }$$
Or $\sum_{j=1}^{e_i} v(\lambda_{\sigma,j}) = v(u_\sigma)$, et $\sum_{\sigma \in S_i} v(u_\sigma) = $deg $L$. Puisque le degré de $L$ est supérieur ou égal à $c$, on obtient que la norme de $\pi_\mathfrak{m}^\kappa : H^0 (\mathcal{U},p_2^* \omega^\kappa) \to H^0(\mathcal{U},p_1^*\omega^\kappa)$ est inférieure à $p^{-c \inf_{\tau \in \Sigma_i} k_\tau}$, pour tout ouvert $\mathcal{U}$ de $C_{\mathfrak{m},rig}$.
\end{proof}

\section{Classicité de formes surconvergentes} \label{classicite}

Cette partie est consacrée à la démonstration du théorème principal.

\begin{theo} \label{maintheorem}
Soit $f$ une forme modulaire surconvergente de poids $\kappa \in \mathbb{Z}^\Sigma$. Supposons que $f$ soit propre pour les opérateurs de Hecke $U_{\pi_i}$, de valeurs propres $a_i$, et que l'on ait pour tout $1\leq i \leq g$
$$e_i(v(a_i) + f_i) < \inf_{\sigma \in \Sigma_i} k_\sigma$$
Alors $f$ est classique.
\end{theo}

Une forme modulaire est définie sur un espace du type $X_{\geq (f_1-\epsilon, \dots, f_g-\epsilon)}$, pour un certain $\epsilon >0$. Pour montrer que $f$ est classique, nous allons tout d'abord prolonger $f$ à tout $X_{rig}$. Le prolongement se fera direction par direction, c'est-à-dire que l'on prolongera $f$ à \\
$Deg^{-1} ([0,f_1] \times [f_2 - \epsilon,f_2] \times \dots \times [f_g - \epsilon,f_g])$, puis à \\
$Deg^{-1} ([0,f_1] \times [0,f_2] \times [f_3 - \epsilon,f_3] \times \dots \times [f_g -\epsilon,f_g])$, et ainsi de suite. \\
Chacune de ses étapes se démontrant de manière analogue, nous ne détaillerons que la première, c'est-à-dire le prolongement à $Deg^{-1} ([0,f_1] \times [f_2 - \epsilon,f_2] \times \dots \times [f_g - \epsilon,f_g])$.

\subsection{Prolongement automatique}

Soit $f$ une forme modulaire surconvergente vérifiant les hypothèses du théorème $\ref{maintheorem}$. Elle est donc définie sur $X_{\geq (f_1-\epsilon, \dots, f_g-\epsilon)}$, pour un certain $\epsilon >0$. Pour tout intervalle $I$, notons $\mathcal{U}_I :=Deg^{-1} (I \times [f_2 - \epsilon,f_2] \times \dots \times [f_g - \epsilon,f_g])$. La forme $f$ est donc définie sur $\mathcal{U}_{[f_1-\epsilon,f_1]}$. Nous allons prolonger $f$ à $\mathcal{U}_{]f_1-\frac{1}{e_1},f_1]}$. 

\begin{prop}
Il est possible de prolonger $f$ à $\mathcal{U}_{]f_1-\frac{1}{e_1},f_1]}$.
\end{prop}

\begin{proof}
Soit $\alpha$ un rationnel avec $0<\alpha<\frac{1}{e_1}$. D'après la proposition $\ref{dyna}$, il existe une entier $N$ tel que 
$$U_{\pi_1}^N (\mathcal{U}_{[f_1-\alpha,f_1]}) \subset \mathcal{U}_{[f_1-\epsilon,f_1]}$$
La fonction $f_\alpha = a_1^{-N} U_{\pi_1}^N f$ est donc définie sur $\mathcal{U}_{[f_1-\alpha,f_1]}$, et est égale à $f$ sur $\mathcal{U}_{[f_1-\epsilon,f_1]}$. Nous noterons donc encore $f$ cette fonction. De plus, les $(\mathcal{U}_{[f_1-\alpha,f_1]})$ pour $0<\alpha<\frac{1}{e_1}$ forment un recouvrement admissible de $\mathcal{U}_{]f_1-\frac{1}{e_1},f_1]}$. On peut donc étendre $f$ à ce dernier intervalle.
\end{proof}

\begin{rema}
Pour démontrer cette proposition, nous avons seulement utilisé le fait que la valeur propre $a_1$ était non nulle.
\end{rema}

\subsection{Séries de Kassaei}

Dans cette partie, nous prolongeons la forme $f$ à $\mathcal{U}_{[0,f_1]}$. Comme les itérés de l'opérateur $U_{\pi_1}$ n'augmentent pas strictement le degré de $H_1$ sur cet ouvert, la méthode de la partie précédente ne s'applique pas. Nous allons construire des séries $f_n$, analogues de celles introduites par Kassaei dans \cite{Ka}, qui convergeront vers $f$. Pour cela, nous utiliserons la décomposition de l'opérateur $U_{\pi_1}$ réalisée dans la partie $\ref{decompo}$. \\
Soit $\epsilon$ un réel strictement positif tel que $v(a_1) + f_1 < (\frac{1}{e_1}-\epsilon) \inf_{\sigma \in \Sigma_1} k_\sigma$. Cela est possible d'après les hypothèses du théorème $\ref{maintheorem}$. Soit $r$ un nombre rationnel avec $f_1-\frac{1}{e_1} < r < f_1-\frac{1}{e_1}+\epsilon$, et $\mathcal{U}:=\mathcal{U}_{[0,r]}$. \\
Soit $N \geq 1$ un entier ; d'après le théorème $\ref{bigtheo}$, on peut trouver une stratification $(\mathcal{U}_j)_{j \in S_N}$ de $\mathcal{U}$, et une décomposition de $U_{\pi_1}^N$ sur chaque strate $\mathcal{U}_j \backslash \mathcal{U}_{j+1}$. De plus, il est possible de faire surconverger arbitrairement cette stratification. En effet, soit $(r^{(k)})$ une suite strictement croissante de rationnels avec $r^{(0)}=r$, $r^{(k)} < f_1-\frac{1}{e_1}+\epsilon$ pour tout $k$, et $(\mathcal{U}_j^{(k)})$ la stratification correspondante à $r^{(k)}$. Alors $\mathcal{U}_j^{(k+1)}$ est un voisinage strict de $\mathcal{U}_j^{(k)}$ dans $\mathcal{U}$ pour tout $j,k$. \\
Notons $\mathcal{V}_j = \mathcal{U}_j^{(j-1)}$ pour tout $j \geq 1$, et $\mathcal{V}_j'=\mathcal{U}_j^{(j)}$ pour tout $i \geq 0$. Alors $\mathcal{V}_j'$ est un voisinage strict de $\mathcal{V}_j$ dans $\mathcal{U}$. Nous avons décomposé l'opérateur $U_{\pi_1}^N$ sur $\mathcal{V}_j' \backslash \mathcal{V}_{j+1}$ en 
$$U_{\pi_1}^N = \coprod_{k=0}^{N-1} U_{\pi_1}^{N-1-k} T_k \coprod T_N$$
avec $T_0 = U_{\pi_1,j}^{good}$ et pour $0 < k < N$
$$T_k = \coprod_{j_1 \in S_{N-1}, \dots, j_k \in S_{N-k}}  U_{\pi_1,j_k}^{good} U_{\pi_1,j_{k-1},j_k}^{bad} \dots U_{\pi_1,j,j_1}^{bad}$$
et
$$T_N = \coprod_{j_1 \in S_{N-1}, \dots, j_{N-1} \in S_{1}} U_{\pi_1,j_{N-1}}^{bad} U_{\pi_1,j_{N-2},j_{N-1}}^{bad} \dots U_{\pi_1,j,j_1}^{bad}$$  
Les images de $U_{\pi_1,j}^{good}$ et de $U_{\pi_1,j_k}^{good}$ ($j_k \in S_{N-k}$) sont incluses dans $\mathcal{U}_{[r^{(j)},f_1]} \subset \mathcal{U}_{[r,f_1]}$, et les opérateurs $U_{\pi_1,i,j}^{bad}$, $U_{\pi_1,j}^{bad}$ ne font intervenir que des supplémentaires $L$ de degré supérieur à $f_1 -r^{(j)} > \frac{1}{e_1} - \epsilon$. 

\begin{defi}
Les séries de Kassaei sur $\mathcal{V}_j' \backslash \mathcal{V}_{j+1}$ sont définies par
$$f_{N,j} := a_1^{-1} U_{\pi_1,j}^{good} f + \sum_{k=1}^{N-1} \sum_{j_1 \in S_{N-1}, \dots, j_k \in S_{N-k}} a_1^{-k-1} U_{\pi_1,j,j_1}^{bad} \dots U_{\pi_1,j_{k-1},j_k}^{bad} U_{\pi_1,j_k}^{good} f$$
\end{defi}

Cette fonction est bien définie, puisque les opérateurs $U_{\pi_1,j}^{good}$ sont soit nuls, auquel cas leur action sur $f$ donne $0$, soit à valeurs dans $\mathcal{U}_{[r,f_1]}$ et $f$ est définie sur cet espace. Ce dernier espace étant quasi-compact, $f$ y est bornée, disons par $M$. \\
La proposition $\ref{lemnorm}$ permet de majorer la norme des opérateurs $a_1^{-1} U_{p,j,k}^{bad}$ : la norme de ces opérateurs est inférieure à
$$u_0 = p^{f_1+v(a_1) - (\frac{1}{e_1}-\epsilon) \inf_{\sigma \in \Sigma_1} k_{\sigma}} < 1$$

\begin{lemm}
Les fonctions $f_{N,i}$ sont uniformément bornées.
\end{lemm}

\begin{proof}
On a 
$$ |a_1^{-k-1} U_{\pi_1,j,j_1}^{bad} \dots U_{\pi_1,j_{k-1},j_k}^{bad} U_{\pi_1,j_k}^{good} f |_{\mathcal{V}_j' \backslash \mathcal{V}_{j+1}} \leq u_0^k |a_1^{-1} U_{\pi_1,j_k}^{good} f |_{U_{\pi_1,j_{k-1},j_k}^{bad} \dots U_{\pi_1,j,j_1}^{bad} (\mathcal{V}_j' \backslash \mathcal{V}_{j+1} )} \leq |a_1^{-1} | p^{f_1}  M$$
car la norme de $U_{\pi_1,j_k}^{good}$ est majorée par $p^{f_1}$.
On peut donc majorer la fonction $f_{N,j}$ par
$$ |f_{N,j} |_{\mathcal{V}_j' \backslash \mathcal{V}_{j+1}} \leq |a_1^{-1} | p^{f_1} M$$
ce qui prouve que les fonctions $f_{N,j}$ sont uniformément bornées. 
\end{proof}

Puisque ces fonctions sont bornées, nous pouvons supposer qu'elles sont entières, quitte à multiplier $f$ par une constante. Nous allons maintenant recoller ces fonctions sur $\mathcal{U}$.

\begin{lemm}
Soient $j,k \in S_N$ et $x \in (\mathcal{V}_j' \backslash \mathcal{V}_{j+1}) \cap (\mathcal{V}_k' \backslash \mathcal{V}_{k+1})$. Alors
$$ | (f_{N,j} -  f_{N,k}) (x) | \leq u_0^N M $$
\end{lemm}

\begin{proof}
La série de Kassaei évalue la fonction $f$ en certains points de $U_{\pi_1}^N(x)$ avec les règles suivantes : si un point est dans $\mathcal{U}_{[f_1-\frac{1}{e_1}+\epsilon,f_1]}$, il est toujours pris en compte, s'il est dans $\mathcal{U}_{[0,r]}$ il n'est jamais pris en compte. La différence entre deux séries ne peux donc porter que sur des points de $U_{\pi_1}^N(x)$ dont le degré de $H_1$ est compris entre $r$ et $f_1-\frac{1}{e_1}+ \epsilon$. \\
De manière plus précise, supposons que $x=(A,i,\phi,H,\omega_{A,\sigma,j})$. Alors ils existent un entier $k\geq 0$, et pour tout $1 \leq i \leq k$ un élément $\epsilon_i \in \{-1,1\}$, des sous-groupes $L_{i,1} \in A[\pi_1]$, $L_{i,l+1} \subset (A/L_{i,l}) [\pi_1]$ qui sont des supplémentaires de l'image de $H_1$, et tels que pour toute section non nulle $\omega$ de $\omega^\kappa$, on ait
$$f_{N,j}(x,\omega) - f_{N,k}(x,\omega) =p^{-Nf_1} a_1^{-N} \sum_{i=1}^k \epsilon_i f(A/L_{i,N},i',\phi',H', \omega_{i,N})$$
où la suite $(\omega_{i,l})$ est déterminée par l'équation $\pi_{i,l}^* \omega_{i,l+1} = \omega_{i,l}$ avec $\pi_{i,l} : A/L_{i,l} \to~A/L_{i,l+1}$ (en posant $\omega_{i,0}=\omega$ et $L_{i,0}=0$). \\
De plus, on a deg $L_{i,l} > \frac{1}{e_1} - \epsilon$ pour tous $i,l$ et deg $L_{i,N} \leq \frac{1}{e_1} - r$ pour tout $i$. \\
Le calcul sur les normes des opérateurs de Hecke (lemme $\ref{lemnorm}$) montre que l'on a 
$$ | (f_{N,i} -  f_{N,j}) (x) | \leq p^{Nf_1 + Nv(a_1) - N(\frac{1}{e_1}-\epsilon) \inf k_{g,i}} |f|_{\mathcal{U}_{[r,f_1]}} \leq u_0^N M$$
\end{proof}

\begin{prop}
Il existe un entier $A_N$ telle que les fonctions $(f_{N,j})_{j\in S_N}$ se recollent en une fonction $g_N \in H^0(\mathcal{U}, \tilde{\omega}^\kappa / p^{A_N})$.
\end{prop}

\begin{proof}
La décomposition de l'ouvert $\mathcal{U}$ étant finie, soit $L$ tel que $\mathcal{V}_{L+1}$ soit vide. La fonction $f_{N,L}$ est donc définie sur $\mathcal{V}_L'$. La fonction $f_{N,L-1}$ est définie sur $\mathcal{V}_{L-1}' \backslash \mathcal{V}_L$. \\
De plus, d'après le lemme précédent, on a
$$ | f_{N,L-1} - f_{N,L} |_{(\mathcal{V}_L' \cap \mathcal{V}_{L-1}') \backslash \mathcal{V}_L} \leq u_0^N M$$
Soit $A_N$ tel que $u_0^N M \leq p^{-A_N}$ ; comme $u_0 < 1$, la suite $(A_N)$ tend vers l'infini. \\
Les fonctions $f_{N,L-1}$ et $f_{N,L}$ sont donc égales modulo $p^{A_N}$ sur $(\mathcal{V}_L' \cap \mathcal{V}_{L-1}') \backslash \mathcal{V}_L$. Comme $(\mathcal{V}_L' \cap \mathcal{V}_{L-1}' , \mathcal{V}_{L-1}' \backslash  \mathcal{V}_L)$ est un recouvrement admissible de $\mathcal{V}_{L-1}'$ , celles-ci se recollent en une fonction $g_{N,L-1} \in H^0 ( \mathcal{V}_{L-1}', \tilde{\omega}^\kappa / p^{A_N})$. \\
De même, $g_{N,L-1}$ et $f_{N,L-2}$ sont égales (modulo $p^{A_N}$) sur $(\mathcal{V}_{L-2}' \cap \mathcal{V}_{L-1} ') \backslash \mathcal{V}_{L-1}$, et donc se recollent en $g_{N,L-2} \in H^0 (\mathcal{V}_{L-2}' , \tilde{\omega}^\kappa / p^{A_N})$. \\
En répétant ce processus, on voit que les fonctions $f_{N,j}$ se recollent toutes modulo $p^{A_N}$ sur $\mathcal{V}_0' = \mathcal{U}$, et définissent donc une fonction $g_N \in H^0(\mathcal{U}, \tilde{\omega}^\kappa / p^{A_N})$.
\end{proof}

\begin{prop}
Les fonctions $(g_N)$ définissent un système projectif dans $\underset{\leftarrow}\lim $~$ H^0 ( \mathcal{U}, \tilde{\omega}^\kappa / p^m)$.
\end{prop}

\begin{proof}
Nous allons prouver que $g_{N+1}$ et $g_N$ sont égales modulo $p^{A_N}$. Soit $ x \in \mathcal{U}$ ; nous avons construit en $x$ les séries de Kassaei $f_{N,j}$ et $f_{N+1,k}$. Or le terme $f_{N+1,k}$ provient d'une décomposition de $U_{\pi_1}^{N+1}$ du type
$$U_{\pi_1}^{N+1} = \sum_{l=0}^{N} U_{\pi_1}^{N-l} T_N + T_{N+1}$$
Nous pouvons donc écrire $f_{N+1,k} = h_1 + h_2$, la fonction $h_1$ étant associée à l'opérateur $\sum_{l=0}^{N-1} U_{\pi_1}^{N-1-l} T_N$  
et $h_2$ à $T_N$. \\
Or la fonction $h_1$ est en réalité une série de Kassaei pour une certaine décomposition de $U_{\pi_1}^N$ : le lemme précédent donne donc 
$$ | (f_{N,j} - h_1) (x) | \leq p^{-A_N}$$
De plus, on a
$$h_2 = \sum_{j_1 \in S_N, \dots, j_N \in S_1} a_1^{-N-1} U_{\pi_1,j,j_1}^{bad} \dots U_{\pi_1,j_{N-1},j_N}^{bad} U_{\pi_1,j_N}^{good} f $$
donc comme les opérateurs $a_{1}^{-1} U_{\pi_1,i,l}^{bad}$ ont une norme inférieure à $u_0$, 
$$ | h_2(x) | \leq u_0^N p^{f_1} |a_1^{-1} | M \leq p^{-A_N'}$$
avec $A_N'=A_N - f_1 - v(a_1)$. Quitte à remplacer $A_N$ par $A_N'$, on voit donc que la réduction de $g_{N+1}$ modulo $p^{A_N}$ est égal à $g_N$.
\end{proof}

En utilisant le gluing lemma (lemme $\ref{glue}$), on voit donc que les fonctions $g_N$ définissent une fonction $g \in H^0(\mathcal{U},\omega^\kappa)$. Bien sûr, $g$ coïncide avec $f$ sur $\mathcal{U}_{]f_1-\frac{1}{e_1},f_1]}$. \\
En effet, si $x \in \mathcal{U}_{]f_1-\frac{1}{e_1},f_1]}$, il existe $N_0$ tel que $U_{\pi_1}^{N} (x) \subset \mathcal{U}_{[f_1-\epsilon,f_1]}$ pour $N \geq N_0$, et la série de Kassaei est alors stationnaire égale à 
$$ a_1^{-N_0} U_{\pi_1}^{N_0} f=f$$
Nous pouvons donc étendre $f$ à $\mathcal{U}_{[0,f_1]}$.

\subsection{Fin de la démonstration}

Nous avons étendu $f$ à $\mathcal{U}_{[0,f_1]} = Deg^{-1} ([0,f_1] \times [f_2-\epsilon,f_2] \times \dots \times [f_g-\epsilon,f_g])$. En utilisant le fait que $f$ soit propre pour $U_{\pi_2}$, et en utilisant la relation vérifiée par la valeur propre $a_2$, la même méthode montre que l'on peut étendre $f$ à $Deg^{-1} ([0,f_1] \times [0,f_2] \times [f_3-\epsilon,f_3] \times \dots \times [f_g-\epsilon,f_g])$. En répétant ce processus, on voit donc que l'on peut étendre $f$ à tout $X_{rig}$. Comme $d>1$, le principe de Koecher et GAGA permettent de montrer que
$$H^0(X_K,\omega^\kappa) = H^0(X_{rig},\omega^\kappa)$$
ce qui permet de conclure que $f$ est classique. Nous détaillons ces résultats dans la partie suivante.

\section{Compactifications et principe de Koecher} \label{compact}

\subsection{Compactification toroïdales}

Dans \cite{Ra}, Rapoport a construit des compactifications de la variété de Hilbert sans niveau, ainsi que pour le niveau $\Gamma(N)$. Sa méthode est très générale, et s'adapte à d'autres structures de niveau. Mentionnons que la construction de compactifications dans le cas de structure de niveau $\Gamma_1(N)$ a été fait dans \cite{Di}. \\
Fixons un idéal $\mathfrak{c}$ un idéal de $O_F$ ; nous allons définir les pointes $\mathfrak{c}$-polarisées. Si $\mathfrak{a}$ est un idéal, on note $\mathfrak{a}^* = \delta^{-1} \mathfrak{a}^{-1}$.

\begin{defi}
Une $(R,N,\pi)$-pointe $\mathcal{C}$ est une classe d'équivalence de couples $(\mathfrak{a},\mathfrak{b},L,\lambda,\beta,H)$ où
\begin{itemize}
\item $\mathfrak{a}$ et $\mathfrak{b}$ sont deux idéaux avec $\mathfrak{a}^* \mathfrak{b} = \mathfrak{c}^*$.
\item $L$ est un réseau de $F^2$ avec une suite exacte
$$0 \to \mathfrak{a}^* \to L \to \mathfrak{b} \to 0$$
\item $\lambda : \wedge^2 L \to \mathfrak{c}^*$ est un isomorphisme $O_F$-linéaire (polarisation).
\item $\beta : N^{-1} \delta^{-1} / \delta^{-1} \hookrightarrow N^{-1} L / L$ est un morphisme injectif.
\item $H$ est un sous-groupe de $\pi^{-1} L /L$ de rang $p^f$, tel que $H = \prod_i H_i$, avec $H_i$ un sous-groupe de ${\pi_i}^{-1} L /L$ de rang $p^{f_i}$, pour $1 \leq i \leq g$. 
\end{itemize}
Des couples $(\mathfrak{a},\mathfrak{b},L,\lambda,\beta,H)$ $(\mathfrak{a}',\mathfrak{b}',L',\lambda',\beta',H')$ sont équivalents s'il existe $\xi \in F$ avec
$\mathfrak{a}'=\xi \mathfrak{a}$, $\mathfrak{b}'=\xi \mathfrak{b}$, un isomorphisme $f : L\simeq L'$ respectant les suites exactes définissant $L$ et $L'$ tel que
\begin{itemize}
\item L'isomorphisme $\wedge^2 L \simeq \wedge^2 L'$ induit un automorphisme de $\mathfrak{c}^*$ donné par un élément de $O_F^{\times,+}$ (unités totalement positives).
\item Les structures de niveau pour $L$ et $L'$ sont isomorphes via $f$.
\end{itemize}
\end{defi}

Si $\mathcal{C}=(\mathfrak{a},\mathfrak{b},L,\lambda,\beta,H)$ est une pointe, on note $\mathfrak{b}'$ l'idéal contenant $\mathfrak{b}$, égal à l'image de $\beta$ dans $N^{-1} \mathfrak{b} / \mathfrak{b}$ (pour tout entier $m$, on a une suite exacte $0 \to m^{-1} \mathfrak{a}^* / \mathfrak{a}^* \to m^{-1} L/L \to m^{-1} \mathfrak{b} / \mathfrak{b} \to 0$). Soit $n$ l'exposant du groupe $\mathfrak{b}'/\mathfrak{b}$. \\
De même, soit $\mathfrak{b}_i'$ l'idéal tel que $\mathfrak{b}_i'/\mathfrak{b}$ soit égal à l'image de $H_i$ dans $\pi_i^{-1} \mathfrak{b} / \mathfrak{b}$. Remarquons que $H_i$ est soit isomorphe à $\pi_i^{-1} \mathfrak{a}^* / \mathfrak{a}^*$, soit à $\pi_i^{-1} \mathfrak{b} / \mathfrak{b}$. On notera $\mathfrak{b}''$ le plus petit idéal contenant $\mathfrak{b}'$ et les $\mathfrak{b}_i'$.

\begin{defi}
Une pointe $(\mathfrak{a}',\mathfrak{b}',L',\lambda',\beta',H')$ appartient à la même composante que $\mathcal{C}$ s'il existe $\xi \in F$ avec
$\mathfrak{a}'=\xi \mathfrak{a}$, $\mathfrak{b}'=\xi \mathfrak{b}$, un isomorphisme $f : L\simeq L'$ respectant les suites exactes définissant $L$ et $L'$, la polarisation, la structure de niveau $\pi$, et tel qu'il existe un automorphisme $\phi$ de $N^{-1} L/L$ induisant l'identité sur $N^{-1} \mathfrak{a}^* / \mathfrak{a}^*$, et la multiplication par $a \in (\mathbb{Z}/n\mathbb{Z})^\times$ sur $\mathfrak{b}'/\mathfrak{b}$, avec
 $$\beta' = f \circ \phi \circ \beta$$
\end{defi}

On notera $\overline{\mathcal{C}}$ la classe de la composante de $\mathcal{C}$. \\
Notons également
$$O_{\mathcal{C},1}^\times = \{ u \in O_F^\times | u \in (1+\mathfrak{b}\mathfrak{b}'^{-1}) \cap (1+ N \mathfrak{b}' \mathfrak{b}^{-1}) \}$$
$$O_{\overline{\mathcal{C}},1}^\times = \{ u \in O_F^\times | u \in ((\mathbb{Z}/n\mathbb{Z})^\times+\mathfrak{b}\mathfrak{b}'^{-1}) \cap (1+ N \mathfrak{b}' \mathfrak{b}^{-1}) \}$$
Soit $H_{\mathcal{C},1} = O_{\overline{\mathcal{C}},1}^\times / O_{\mathcal{C},1}^\times$. Nous utilisons ces notations pour être cohérent avec \cite{Di}. \\
Nous allons maintenant définir les cartes locales pour la pointe $\mathcal{C}$. Soit $P= \mathfrak{a} \mathfrak{b}''$ (noté $X$ dans \cite{Di}), $S=S_\mathcal{C} = \mathbb{G}_m \otimes P^*$. Soit $\Sigma^\mathcal{C}$ un éventail complet de $P^*_+$, l'ensemble des éléments totalement positifs de $P^*$. Soit $\sigma \in \Sigma^\mathcal{C}$. On peut lui associer des espaces $S_\sigma$, $\overline{S}_\sigma$ et $\overline{S}_\sigma^0$ (voir \cite{Di}) ; on notera $S \hookrightarrow S_{\Sigma^\mathcal{C}}$ l'immersion torique obtenue en recollant les immersions $S \hookrightarrow S_\sigma$, et $S_{\Sigma^\mathcal{C}}^\wedge$ la complétion de $S_{\Sigma^\mathcal{C}}$ le long de $S_{\Sigma^\mathcal{C}} \backslash S$. La construction de Mumford appliquée à $(P,\mathfrak{a},\mathfrak{b})$ donne alors un schéma semi-abélien $G_\sigma$ sur $\overline{S}_\sigma$, muni d'une action de $O_F$, et dont la restriction à $\overline{S}_\sigma^0$ est un SAHB muni d'une $\mathfrak{c}$-polarisation que l'on notera $G_\sigma^0$. De plus, pour tout idéal $\mathfrak{m}$, on a une suite exacte
$$0 \to (\mathfrak{a} / \mathfrak{m} \mathfrak{a}) (1) \to G_\sigma^0 [\mathfrak{m}] \to \mathfrak{m}^{-1} \mathfrak{b} / \mathfrak{b} \to 0$$
où $(1)$ désigne le dual de Cartier d'un schéma en groupes. \\
Nous allons maintenant associer à $G_\sigma^0$ des structures de niveau. La structure de niveau $\Gamma_1(N)$ a été faite dans \cite{Di}. Remarquons qu'obtenir cette structure de niveau nécessite d'uniformiser la pointe, et de se placer sur $\mathbb{Z}[\frac{1}{N},\zeta_\mathcal{C}]$, où $\zeta_{\mathcal{C}}$ est une racine de l'unité d'ordre $n$, l'exposant de $\mathfrak{b}'/\mathfrak{b}$. Décrivons comment obtenir la structure de niveau $\Gamma_0(\pi)$. Cela revient à se donner un sous-groupe $H_i^0$ de $G_\sigma^0[\pi_i]$ de rang $p^{f_i}$ pour tout $i$. \\
Si le sous-groupe $H_i$ relatif à la pointe $\mathcal{C}$ est égal à $\pi_i^{-1} \mathfrak{a}^* / \mathfrak{a}^*$, on définit $H_i^0$ comme l'image de $(\mathfrak{a} / \pi_i \mathfrak{a})(1)$ dans $G_\sigma^0[\pi_i]$. Sinon, $H_i$ est isomorphe à $\pi_i^{-1} \mathfrak{b} / \mathfrak{b} = \mathfrak{b}_i' / \mathfrak{b}$. La construction de Mumford appliquée à $(P,\mathfrak{a},\mathfrak{b}_i')$ donne un schéma semi-abélien $G_{\sigma,i}'$ sur $\overline{S}_\sigma$, muni d'une action de $O_F$, et dont la restriction à $\overline{S}_\sigma^0$ est un SAHB muni d'une $\mathfrak{c}_i'=\mathfrak{a} \mathfrak{b}_i'$-polarisation que l'on notera $G_{\sigma,i}^{'0}$. Par fonctorialité, on a une isogénie $G_{\sigma}^0 \to G_{\sigma,i}^{'0}$. On dispose alors d'une suite exacte
$$0 \to \mathfrak{b}_i' / \mathfrak{b} \to G_{\sigma}^0 [\pi_i] \to G_{\sigma,i}^{'0} [\pi_i] \to 0$$
On définit alors $H_i^0$ comme l'image de $\mathfrak{b}_i'/ \mathfrak{b}$ dans $G_{\sigma}^0 [\pi_i]$. \\
De plus, les variétés abéliennes données par la construction de Mumford étant ordinaires, le faisceau conormal de $G_{\sigma}^0$ est un $O_F \otimes_{\mathbb{Z}} \mathcal{O}_{\overline{S}_\sigma^0}$-module localement libre de rang $1$, donc est canoniquement filtré. Nous avons donc défini un point de $X_\mathfrak{c}$. Plus précisément, il existe un diagramme cartésien
\begin{displaymath}
\xymatrix{
G_{\sigma}^0 \times \text{Spec}(O_K[\zeta_\mathcal{C}]) \ar[r] & A \ar[d] \\
\overline{S}_\sigma^0 \times \text{Spec}(O_K[\zeta_\mathcal{C}]) \ar[r] & X_\mathfrak{c}
}
\end{displaymath}
où $A$ est le SAHB universel sur $X_\mathfrak{c}$. \\
Nous allons recoller les cartes locales avec $X_\mathfrak{c}$, de manière à obtenir une compactification de ce dernier espace.

\begin{defi}
Un éventail admissible $\Sigma=(\Sigma^\mathcal{C})_\mathcal{C}$ est la donnée pour chaque classe d'isomorphisme de $(R,N,\pi)$-composante $\mathcal{C}$ d'un éventail complet $\Sigma^\mathcal{C}$ de $P^*_+$, stable par $O_{\mathcal{C},1}^\times$, et ne contenant qu'un nombre fini d'élément modulo cette action.
\end{defi}

Voici l'analogue du principal théorème de \cite{Ra} et de \cite{Di}.

\begin{theo}
Soit $\Sigma=(\Sigma^\mathcal{C})_\mathcal{C}$ un éventail admissible. Alors il existe un Spec$(O_K)$-schéma $X_\mathfrak{c}^\Sigma$, une immersion ouverte $X_\mathfrak{c} \hookrightarrow X_\mathfrak{c}^\Sigma$, et un isomorphisme de schémas formels
$$ \coprod_\mathcal{C} (S_{\Sigma^\mathcal{C}}^\wedge / O_{\mathcal{C},1}^\times ) \times \text{Spec}(O_K[\zeta_\mathcal{C}]^{H_{\mathcal{C},1}}) \simeq X_\mathfrak{c}^{\Sigma \wedge}   $$
où $X_\mathfrak{c}^{\Sigma \wedge}$ désigne le complété formel de $X_\mathfrak{c}^\Sigma$ le long de sa partie à l'infini.
\end{theo}

\begin{proof}
La démonstration est analogue à celle du théorème $7.2$ de \cite{Di}. Celle-ci utilise le théorème de géométrie rigide démontré par Rapoport dans \cite{Ra} (théorème $3.5$). La vérification des hypothèses de ce théorème utilise la construction de Raynaud (voir dans \cite{Di} par exemple), qui décrit les variétés abéliennes sur $L$, un corps de fractions d'un anneau de valuation discrète, qui ont mauvaise réduction semi-stable déployée. \\
Décrivons rapidement les deux conditions à vérifier pour appliquer le théorème de Rapoport. La première demande d'étudier les $L$-points de $\overline{S}_{\sigma_1}^0$ et $\overline{S}_{\sigma_2}^0$ qui donnent la même variété abélienne $A$ sur $L$, où $\sigma_j$ est un cône d'une certaine composante $\mathcal{C}_j$. La description de Raynaud, ainsi que les structures de niveau sur $A$ montrent que les deux composantes $\mathcal{C}_1$ et $\mathcal{C}_2$ sont isomorphes, et que les cônes $\sigma_1$ et $\sigma_2$ sont d'intersection non vide. Cela permet de vérifier le premier point. \\
Pour la seconde hypothèse, il faut vérifier que si on a des morphismes $f : $ Spec $L \to \overline{S}_\sigma^0$, et $g :$ Spec $ R \to X_\mathfrak{c}$ ($R$ est l'anneau des entiers de $L$), compatible par le morphisme Spec $L \to $ Spec $V$, alors on peut étendre $f$ à Spec $R$. Or ces morphismes donnent des SAHB $A$ sur $L$, et $A'$ sur $R$ tels que $A = A' \otimes_R L$. Le morphisme $f$ donne une composante $\mathcal{C}$, et la construction de Mumford permet de décrire la variété abélienne $A$. la description de Raynaud permet quant à elle de décrire la variété abélienne $A'$. La compatiblité entre la construction de Mumford et celle de Raynaud permettent alors d'identifier ces deux descriptions, et d'étendre le morphisme $f$ à Spec $R$.
\end{proof}

Nous noterons $\overline{X_\mathfrak{c}} = X_\mathfrak{c}^\Sigma$, même si la construction de cet espace dépend d'un choix de $\Sigma$. Enonçons quelques propriétés de cet espace.

\begin{prop}
Il existe un unique schéma en groupes semi-abélien $\mathcal{A} \to \overline{X_\mathfrak{c}}$ qui étende le SAHB universel $A$ sur $X_\mathfrak{c}$. Il est muni d'une action de $O_F$, et c'est un tore sur $\overline{X_\mathfrak{c}} \backslash X_\mathfrak{c}$.
\end{prop}

\begin{proof}
L'existence de $\mathcal{A}$ découle de la construction, puisque nous avons construit un schéma semi-abélien $G_\sigma \times$Spec$(O_K[\zeta_\mathcal{C}])$ sur $\overline{S}_\sigma \times$Spec$(O_K[\zeta_\mathcal{C}])$ pour toute pointe $\mathcal{C}$. \\
Pour démontrer l'unicité, remarquons tout d'abord que le bord $\overline{X_\mathfrak{c}} \backslash X_\mathfrak{c}$ est inclus dans le lieu de Rapoport. Il suffit donc de démontrer l'unicité du prolongement de $X_\mathfrak{c}^R$ à $\overline{X_\mathfrak{c}}^R$, où $X_\mathfrak{c}^R$ est l'ouvert de Rapoport de $X_\mathfrak{c}$, et de même pour $\overline{X_\mathfrak{c}}$. Le schéma $\overline{X_\mathfrak{c}}^R$ est lisse donc normal, et $X_\mathfrak{c}^R$ est un ouvert dense de cet espace. On peut alors appliquer la proposition $2.7$ de \cite{F-C}  pour conclure.
\end{proof}

\begin{prop}
Le schéma $\overline{X_\mathfrak{c}}$ est propre sur $O_K$.
\end{prop}

\begin{proof}
Il suffit de vérifier le critère valuatif de propreté. Soit $R$ un anneau de valuation discrète de corps des fractions $L$. Comme $X_\mathfrak{c}$ est un ouvert dense de $\overline{X_\mathfrak{c}}$, il suffit de montrer que tout morphisme $f : $ Spec $L \to X_\mathfrak{c}$ peut se prolonger en un morphisme Spec $R \to \overline{X_\mathfrak{c}}$. \\
Si la variété abélienne $A$ sur $L$ donnée par $f$ est à bonne réduction sur $R$, il existe une variété abélienne $A_0$ sur $R$ telle que $A=A_0 \otimes_R L$. La polarisation $\phi$, et les structures de niveau $\mu_N$ et $\Gamma_0(\pi)$ pour $A$ donnent une polarisation et des structures de niveau pour $A_0$. De plus, la filtration de $\omega_A$, qui est un $L$-espace vectoriel, donne une filtration pour $\omega_{A_0}$. Il suffit en effet de prendre l'image inverse de la filtration de $\omega_A$ par le morphisme $\omega_{A_0} \to \omega_{A_0} \otimes_R L = \omega_A$. On peut donc définir un morphisme Spec $R \to X_\mathfrak{c}$ qui étend $f$. \\
Supposons maintenant que la variété abélienne $A$ a mauvaise réduction. La théorie de géomètrie rigide de Raynaud (voir \cite{Di} par exemple) fournit alors deux idéaux $\mathfrak{a}$ et $\mathfrak{b}$ tels que $\mathfrak{c}=\mathfrak{a} \mathfrak{b}^{-1}$, et $A_{rig} = (\mathbb{G}_m \otimes \mathfrak{a}^*)_{rig} / \mathfrak{b}_{rig}$. Les structures de niveau pour $A$ définissent alors une $(R,N,\pi)$-composante $\mathcal{C}$. La description de Raynaud fournit en plus un élément $\xi^* \in (\mathfrak{a} \mathfrak{b})^*_+$. Un translaté de $\xi^*$ par $O_{\mathcal{C},1}^\times$ appartient à un cône $\sigma \in \Sigma^\mathcal{C}$ utilisé pour la construction de $\overline{X_c}$. Le morphisme $f$ se factorise donc par la carte locale $\overline{S}_\sigma^0 \times \text{Spec}(O_K[\zeta_\mathcal{C}]) \to X_\mathfrak{c}$. Le morphisme Spec $L \to~\overline{S}_\sigma^0 \times~\text{Spec}(O_K[\zeta_\mathcal{C}]) \hookrightarrow~\overline{S}_\sigma \times~\text{Spec}(O_K[\zeta_\mathcal{C}])$ s'étend alors nécessairement en un morphisme Spec $R \to \overline{S}_\sigma \times~\text{Spec}(O_K[\zeta_\mathcal{C}])$. Le morphisme
$$g : \text{Spec } R \to \overline{S}_\sigma \times \text{Spec}(O_K[\zeta_\mathcal{C}]) \to \overline{X_\mathfrak{c}}$$
étend alors le morphisme $f$.
\end{proof}

Nous noterons $\overline{X} = \coprod_i \overline{X_{\mathfrak{c}_i}}$ ; c'est encore un schéma propre sur $O_K$.

\subsection{Principe de Koecher}

Comme il existe un schéma semi-abélien $\mathcal{A}$ sur $\overline{X}$, le faisceau $\omega_A$ se prolonge en $\omega_\mathcal{A}$ sur $\overline{X}$. De plus, $\mathcal{A}$ est un tore sur le bord $\overline{X} \backslash X$ ; ce dernier espace est donc dans le lieu de Rapoport. 

\begin{prop}
Le faisceau $\omega^\kappa$ se prolonge sur $\overline{X}$.
\end{prop}

Comme le degré de $F$ est supérieur ou égal à $2$, on dispose du principe de Koecher. Nous nous inspirons de la démonstration de \cite{Di}, théorème $8.3$.

\begin{theo}
Pour tout $O_K$-algèbre $R$, on a 
$$H^0(X\times\text{Spec } R, \omega^\kappa) = H^0(\overline{X} \times \text{Spec } R, \omega^\kappa)$$
\end{theo}

\begin{proof}
Soit $f \in H^0(X\times$Spec $R, \omega^\kappa)$ ; nous voulons montrer que $f$ peut s'étendre en un élément de $H^0(\overline{X} \times$Spec $R, \omega^\kappa)$. Il suffit de montrer que $f$ peut s'étendre aux pointes de $X_\mathfrak{c}$, pour tout $\mathfrak{c}$. Soit donc $\mathcal{C}$ une telle pointe, et $P$ l'idéal fractionnaire associé. La forme modulaire $f$ possède un $q$-développement méromorphe le long de $S_{\Sigma^\mathcal{C}}^\wedge$ : 
$$f = \sum_{\xi \in P} a_\xi q^\xi$$
avec $a_\xi \in R$. Il n'existe qu'un nombre fini de $\xi \notin P_+$ avec $a_\xi \neq 0$, et pour tout $u \in O_{\mathcal{C},1}^\times$, on a $a_\xi =0 \Leftrightarrow a_{u^2 \xi} = 0$. \\
Supposons qu'il existe $\xi_0 \notin (P_+ \cup \{0\})$, avec $a_{\xi_0} \neq 0$. Alors $a_{u^2 \xi_0}$ est également non nul pour $u \in O_{\mathcal{C},1}^\times$. Or $u^2 \xi_0$ n'appartient pas à $P_+ \cup \{0\}$ si $u \in O_{\mathcal{C},1}^\times$, et comme $d \geq 2$ et $\xi_0 \neq 0$, l'ensemble $\{u^2 \xi_0, u \in O_{\mathcal{C},1}^\times \}$ est infini d'après le théorème des unités de Dirichlet. On obtient donc une contradiction. \\ 
La forme $f$ a donc un prolongement holomorphe en la pointe $\mathcal{C}$.
\end{proof}

En appliquant ce résultat à $R=O_K/p^n$ pour tout $n$, en prenant la limite, puis en tensorisant par $K$, on obtient que
$$H^0(X_{rig},\omega^\kappa) = H^0(\overline{X}_{rig},\omega^\kappa)$$
où $\overline{X}_{rig}$ est la fibre générique rigide de la complétion formelle de $\overline{X}$ le long de sa fibre spéciale.
De plus, le schéma $\overline{X}$ étant propre sur $O_K$, on a par GAGA (voir \cite{EGA} partie $5.1$)
$$H^0(\overline{X} \times \text{Spec } K, \omega^\kappa) = H^0(\overline{X}_{rig}, \omega^\kappa)$$

En résumé, nous avons le résultat suivant.

\begin{theo}
L'espace des formes modulaires, $H^0(X_K, \omega^\kappa)$, est égal à $H^0(X_{rig},\omega^\kappa)$.
\end{theo}

\backmatter

\end{document}